\documentclass[12pt]{amsart}
\usepackage[all]{xy}

\usepackage{graphicx}

\usepackage{amsmath}
\usepackage{amssymb}
\usepackage{amsfonts}

\newcommand{\Cdb}{\mbox{$\mathbb{C}$}}
\newcommand{\Ddb}{\mbox{$\mathbb{D}$}}

\newcommand{\Rdb}{\mbox{$\mathbb{R}$}}

\newcommand{\Zdb}{\ensuremath{\mathbb{Z}}}

\newcommand{\norm}[1]{\Vert#1\Vert}
\newcommand{\bignorm}[1]{\bigl\Vert#1\bigr\Vert}
\newcommand{\Bignorm}[1]{\Bigl\Vert#1\Bigr\Vert}

\newtheorem{theorem}{Theorem}[section]
\newtheorem{lemma}[theorem]{Lemma}
\newtheorem{corollary}[theorem]{Corollary}
\newtheorem{proposition}[theorem]{Proposition}

\theoremstyle{remark}
\newtheorem{remark}[theorem]{\bf Remark}
\theoremstyle{definition}

\numberwithin{equation}{section}

\begin{document}

\title[]{Strong $q$-variation inequalities for analytic semigroups}

\author{Christian Le Merdy, Quanhua Xu}
\address{Laboratoire de Math\'ematiques\\ Universit\'e de  Franche-Comt\'e
\\ 25030 Besan\c con Cedex\\ France}
\email{clemerdy@univ-fcomte.fr}

\address{School of Mathematics and Statistics\\ Wuhan University\\ 
Wuhan 430072 \\ Hubei\\ China}

\address{Laboratoire de Math\'ematiques\\ Universit\'e de  Franche-Comt\'e
\\ 25030 Besan\c con Cedex\\ France}
\email{qxu@univ-fcomte.fr}

\date{\today}

\thanks{The authors are both supported by the research program ANR-06-BLAN-0015}

\begin{abstract}Let $T\colon L^p(\Omega)\to L^p(\Omega)$ be a positive contraction,
with $1<p<\infty$. Assume that $T$ is analytic, that is, there exists a constant $K\geq 0$ such that $\norm{T^n-T^{n-1}}\leq K/n$ for any integer $n\geq 1$. Let $2<q<\infty$ and 
let $v^q$ be the space of all complex sequences with a finite strong $q$-variation. We show that for
any $x\in L^p(\Omega)$, the sequence $\bigl([T^n(x)](\lambda)\bigr)_{n\geq 0}$
belongs to $v^q$ for almost every $\lambda\in\Omega$, with an estimate
$\norm{ (T^n(x))_{n\geq 0}}_{L^p(v^q)}\leq C\norm{x}_p$. If we remove the
analyticity assumption, we obtain an estimate
$\norm{(M_n(T)x)_{n\geq 0}}_{L^p(v^q)}\leq C\norm{x}_p$, where
$M_n(T)=(n+1)^{-1}\sum_{k=0}^{n} T^k\,$ denotes the ergodic averages of $T$.
We also obtain similar results for strongly continuous semigroups 
$(T_t)_{t\geq 0}$ of positive contractions on $L^p$-spaces.
\end{abstract}

\maketitle


\bigskip

\section{Introduction.}

Variational inequalities in probability, ergodic theory and harmonic analysis
have been the subject of many recent research papers.  One important character of these
inequalities is the fact that they can be used to measure the speed of convergence
for the family of operators in consideration. To be more precise, consider, for instance,
a measure space $(\Omega,\mu)$ and  an operator  $T$ on $L^1(\Omega)+ L^\infty(\Omega)$.  
Form the ergodic averages of $T$:
 $$
M_n(T)\,=\,\frac{1}{n+1}\,\sum_{k=0}^{n} T^k\,,\qquad n\geq 0.
$$
A fundamental theorem  in ergodic theory states that if $T$ is a contraction on
$L^p(\Omega)$ for every $1\le p\le\infty$, then the limit
$\lim_{n\to\infty}M_n(T)x$ exists a.e. for every $x\in L^p(\Omega)$.
One can naturally ask what is the speed of convergence of this limit.
A classical tool for measuring that speed is the following square function
$$
S(x)=\Bigl(\sum_{n\ge0} n\bigl|M_{n+1}(T)x -M_n(T)x \bigr|^2\Bigr)^{\frac12},
$$
the problem being to estimate its norm $\|S(x)\|_p$. This issue goes back to Stein \cite{S2},
who proved that if $T$ as above is positive on $L^2(\Omega)$ (in the Hilbertian sense),
then a square function inequality
$$
\|S(x)\|_p\le C_p\|x \|_p,\quad x\in L^p(\Omega),
$$
holds for any $1<p<\infty$.
Stein's inequality is closely related to Dunford-Schwartz's maximal ergodic inequality,
$$
\big\|\sup_{n\ge0}|M_n(T)x|\,\big\|_p\le C_p\|x\|_p,\quad x\in L^p(\Omega), \; 1<p\le\infty.
$$
This maximal inequality and its weak type $(1,1)$ substitute for $p=1$ are key
ingredients in the proof of the previous pointwise ergodic theorem.

The strong $q$-variation is another (better) tool to measure
the speed of $\lim_{n\to\infty}M_n(T)x$.
Bourgain was the first to consider variational inequalities in ergodic theory.
To state his inequality we need to recall the definition of the strong $q$-variation.
Given a sequence $(a_n)_{n\geq 0}$ of complex numbers and a number $1\leq q<\infty$,
the strong $q$-variation norm is defined as
$$
\norm{(a_n)_{n\geq 0}}_{v^q}\,=\, \sup\Bigl\{\bigl(\vert a_0\vert^q\, +\,\sum_{k\geq 1}
\vert a_{n_k}-a_{n_{k-1}}\vert^q \bigr)^{\frac{1}{q}}\Bigr\},
$$
where the supremum runs over all increasing sequences $(n_k)_{k\geq 0}$
of integers such that $k_0=0$.
It is clear that the set $v^q$ of all sequences with a finite strong $q$-variation is a Banach space
for the norm $\norm{\ }_{v^q}$.

Bourgain \cite{B} proved that if $T$ is induced by a measure preserving transformation
on $(\Omega,\mu)$, then for any $ 2<q<\infty$,
$$
\bignorm{(M_n(T)x)_{n\geq 0}}_{L^2(v^q)}\,
\leq\, C_{q}\,\norm{x}_2, \quad x\in L^2(\Omega).
$$
This inequality was then extended to $L^p(\Omega)$ for any $1<p<\infty$ by Jones,
Kaufman, Rosenblatt and Wierdl \cite{JKRW}. The latter paper contains many other
interesting results on the subject. Note that the predecessor of Bourgain's inequality
is L\'epingle's variational inequality for martingales \cite{Le}. The latter says
that if $(\mathbb E_n)_{n\geq 0}$ is an increasing sequence of conditional expectations
on a probability space $\Omega$, then we have an estimate
$\bignorm{(\mathbb E_n(x))_{n\geq 0}}_{L^p(v^q)}\,
\leq\, C\,\norm{x}_p$
(see also \cite{PX2} for further results on that theme).

Since \cite{B}, variational type inequalities have been extensively studied
in ergodic theory and harmonic analysis. Many classical sequences of operators
and semigroups have been proved to satisfy strong variational bounds, see in particular
\cite{CMMTV, JR, JSW, JW, OSTTW} and references therein.
The main purpose of this paper is to exhibit a large class of operators
$T$ on $L^p(\Omega)$ for a fixed $1<p<\infty$ with the following property:
for any $2<q<\infty$, there exists a constant $C>0$ (which may depend on $q$ and $T$)
such that
for any $x\in L^p(\Omega)$, the sequence $(T^n(x))_{n\geq 0}$
belongs to $L^p(\Omega;v^q)$, and
\begin{equation}\label{1q1}
\bignorm{(T^n(x))_{n\geq 0}}_{L^p(v^q)}\,
\leq\, C\,\norm{x}_p.
\end{equation}
We show that this holds true provided that $T$ is a positive contraction (more generally,
a contractively regular operator) and $T$ is analytic, in the sense that
$$
\sup_{n\geq 1} n\norm{T^n-T^{n-1}}\,<\,\infty.
$$

Inequality  \eqref{1q1} implies, of course, a similar variational inequality
for the ergodic averages $M_n(T)$. However, in this latter case,
the analytic assumption above can be removed.
Namely, for a positive contraction $T$ on $L^p(\Omega)$ with $1<p<\infty$ we have
an estimate
\begin{equation}\label{1q3}
\bignorm{(M_n(T)x)_{n\geq 0}}_{L^p(v^q)}\,\leq\,
C\,\norm{x}_p,\qquad x\in L^p(\Omega),
\end{equation}
for any $2<q<\infty$. This result extends those of \cite{B} and \cite{JKRW} quoted previously.

Note that inequality \eqref{1q1} for positive analytic contractions 
considerably improves our previous maximal ergodic inequality for such operators $T$ proved in \cite{LMX}. 
In this sense, this paper is a continuation of \cite{LMX}. On the other hand,  
our proof of \eqref{1q1} heavily relies on the square function inequality of  \cite{LMX}.

We also establish results similar to (\ref{1q1}) and (\ref{1q3}) 
for strongly continuous semigroups.
This requires the following continuous analog of $v^q$.
Given a complex family $(a_t)_{t>0}$, define
$$
\norm{(a_t)_{t>0}}_{V^q}\,=\, \sup\Bigl\{\bigl(\vert a_{t_0}\vert^q\, +\,\sum_{k\geq 1}
\vert a_{t_k}-a_{t_{k-1}}\vert^q \bigr)^{\frac{1}{q}}\Bigr\},
$$
where the supremum runs over all increasing sequences $(t_k)_{k\geq 0}$
of positive real numbers. Then we let
$V^q$ be the resulting Banach space of
all $(a_t)_{t>0}$ such that
$\norm{(a_t)_{t> 0}}_{V^q}\,<\,\infty\,$.

Consider a bounded analytic semigroup $(T_t)_{t\geq 0}$ on $L^p(\Omega)$, with
$1<p<\infty$. We show that if $T_t$ is a positive contraction (more generally,
a contractively regular operator) for any $t\geq 0$, then for any $2<q<\infty$
and any $x\in L^p(\Omega)$,
the family $\bigl([T_t(x)](\lambda)\bigr)_{t>0}$ belongs to $V^q$ for almost every $\lambda\in
\Omega$, and we have an estimate
\begin{equation}\label{1q2}
\Bignorm{\lambda\mapsto \bignorm{\bigl([T_t(x)](\lambda)\bigr)_{t>0}}_{V^q}}_p
\,\leq\,C\,\norm{x}_p,\qquad x\in L^p(\Omega).
\end{equation}

We mention that Jones and Reinhold \cite{JR} proved  (\ref{1q2}) for positive, unital
symmetric diffusion semigroups and (\ref{1q1}) for a certain class of convolution
operators (only in the case $p=2$). Our results turn out
to extend these contributions in various directions. 

As in the discrete case, we obtain similar results for the averages 
of the semigroup if we remove the analyticity assumption.

Inequalities for ergodic averages, such as (\ref{1q3}), will be established in Section 3.
Then our main results leading to (\ref{1q1}) and (\ref{1q2}) will be established in
Section 4, using the above mentioned results from Section 3, as well as square
function estimates from \cite{LMX}.

Section 5 is devoted to various complements. On the one hand, we establish individual (= pointwise)
ergodic theorems in our context. For example, if $T\colon L^p(\Omega)\to L^p(\Omega)$
is a positive analytic contraction, then $\lim_{n\to\infty} T^n(x)$ exists a.e. for
every $x\in L^p(\Omega)$. On the other hand,
it is well-known that (\ref{1q3}) cannot be extended
to the case $q=2$. Then we show analogs of (\ref{1q3}) and (\ref{1q1}) when $v^q$
is replaced by the oscillation space $o^2$. We give examples and applications
in Section 6.

\bigskip
We end this introduction with a few notation. If $X$ is a Banach space, we let
$B(X)$ denote the algebra of all bounded operators on $X$ and we let
$I_X$ denote the identity operator on $X$ (or simply $I$ if there is no ambiguity
on $X$). For any $T\in B(X)$, we let
$\sigma(T)$ denote the spectrum of $T$. Also we let
$\Ddb=\{z\in\Cdb\, :\, \vert z\vert <1\}$ be the open unit disc of $\Cdb$.

For a measurable function $x\colon\Omega\to\Cdb$ acting on
a measure space $(\Omega,\mu)$, $\norm{x}_p$ denotes the $L^p$-norm of $x$.

We refer to \cite{Pa} and \cite{Go} for background on strongly continuous
and analytic semigroups on Banach space.

\medskip
\section{Preliminaries on $q$-variation.}

The aim of this section is to provide some elementary background on the spaces $v^q$ and
$V^q$ defined above. We fix some $1\leq q<\infty$.

Let $(a_n)_{n\geq 0}$ be an element of $v^q$. Then the sequence $(a_n)_{n\geq 0}$ is bounded, with
$$
\norm{(a_n)_{n\geq 0}}_{\infty} \leq 2
\norm{(a_n)_{n\geq 0}}_{v^q},
$$
and
$$
\lim_{m\to \infty}\bignorm{(0,\ldots,0,a_{m+1} - a_m,\ldots, a_n-a_m,\ldots)}_{v^q}\,=0.
$$
Thus the space of eventually constant sequences is dense in $v^q$.
Consequently any element of $v^q$ is a converging sequence.

For any integer $m\geq 1$, let $v_m^q$ be the space of $(m+1)$-tuples $(a_0,a_1,\ldots,a_m)$
of complex numbers, equipped with the norm
$$
\bignorm{(a_0,a_1,\ldots,a_m)}_{v^q_m}\,=\,\bignorm{(a_0,a_1,\ldots,a_m,a_m,\ldots)}_{v^q}.
$$
It follows from above that an infinite sequence $(a_n)_{n\geq 0}$ belongs to $v^q$ if and only if there is a constant
$C\geq 0$ such that  $\bignorm{(a_0,a_1,\ldots,a_m)}_{v^q_m}\,\leq C$ for any $m\geq 1$ and in this case,
\begin{equation}\label{6Approx}
\norm{(a_n)_{n\geq 0}}_{v^q}\,=\,\lim_{m\to \infty} \bignorm{(a_0,a_1,\ldots,a_m)}_{v^q_m}.
\end{equation}

Let $(\Omega,\mu)$ be a measure space, let
$1<p<\infty$ and let $L^p(\Omega;v^q)$ denote the
corresponding Bochner space. Any element of that space can be naturally regarded
as a sequence of $L^p(\Omega)$. The following is a direct consequence
of the above approximation properties.

\begin{lemma}\label{6Lem1}
Let $(x_n)_{n\geq 0}$ be a sequence of $L^p(\Omega)$, the following assertions
are equivalent.
\begin{itemize}
\item [(i)]
The sequence $(x_n)_{n\geq 0}$ belongs to $L^p(\Omega;v^q)$.
\item [(ii)]
The sequence $(x_n(\lambda))_{n\geq 0}$ belongs to $v^q$ for almost every $\lambda\in\Omega$ and the
function $\lambda\mapsto \norm{(x_n(\lambda))_{n\geq 0}}_{v^q}$ belongs to $L^p(\Omega)$.
\item [(iii)]
There is a constant $C\geq 0$ such that
$$
\bignorm{(x_0,x_1,\ldots,x_m)}_{L^p(v_m^q)}\,\leq C
$$
for any $m\geq 1$.
\end{itemize}
In this case,
$$
\bignorm{\lambda\mapsto \norm{(x_n(\lambda))_{n\geq 0}}_{v^q}}_p\,=\,
\norm{(x_n)_{n\geq 0}}_{L^p(v^q)}\,=\,
\lim_{m\to \infty} \bignorm{(x_0,x_1,\ldots,x_m)}_{L^p(v^q_m)}.
$$
\end{lemma}

\bigskip
We now consider the continuous case. We note for further use that any element $(a_t)_{t>0}$ of $V^q$ admits limits
$$
\lim_{t\to 0^+} a_t\qquad\hbox{and}\qquad \lim_{t\to\infty} a_t.
$$
The following is a continuous analog of Lemma \ref{6Lem1}. Note however that families satisfying the next statement
do not necessarily belong to the Bochner space $L^p(\Omega; V^q)$.

\begin{lemma}\label{6Lem2}
Let $(x_t)_{t>0}$ be a family of $L^p(\Omega)$ and assume that:
\begin{itemize}
\item [(1)] For a.e. $\lambda\in\Omega$, the function $t\mapsto x_t(\lambda)$ is continuous on $(0,\infty)$.
\item [(2)] There exists a constant $C\geq 0$ such that whenever $t_0<t_1<\cdots<t_m$
is a finite increasing sequence of positive real numbers, we have
$$
\bignorm{(x_{t_0},x_{t_1},\ldots,x_{t_m})}_{L^p(\Omega; v^q_m)}\,\leq C.
$$
\end{itemize}
Then $(x_t(\lambda))_{t>0}$ belongs to $V^q$ for a.e. $\lambda\in\Omega$, the function
$\lambda\mapsto \bignorm{(x_t(\lambda))_{t>0}}_{V^q}$ belongs to $L^p(\Omega)$ and
$$
\bignorm{\lambda\mapsto \norm{(x_t(\lambda))_{t>0}}_{V^q}}_{p}\,\leq C.
$$
\end{lemma}

\begin{proof}
For any integer $N\geq 1$, define
$$
\varphi_N(\lambda) \,=\, \sup_{k\geq 1}\bignorm{\bigl(x_{n2^{-N}} (\lambda)\bigr)_{n\geq k}}_{v^q},
\qquad \lambda\in \Omega.
$$
It follows from (\ref{6Approx}) that $\varphi_N$ is measurable. Moreover the sequence
$(\varphi_N)_{N\geq 1}$ is nondecreasing, and we may therefore define
$$
\varphi(\lambda)=\lim_{N\to\infty}\varphi_N(\lambda), \qquad \lambda\in \Omega.
$$
By construction, $\varphi$ is measurable and by the monotone convergence theorem,
its $L^p$-norm is equal to $\lim_N\norm{\varphi_N}_p$.
According to the assumption (2) and the approximation property
(\ref{6Approx}), the $L^p$-norm of $\varphi_N$ is $\leq C$ for any
$N\geq 1$. Hence
$$
\int_{\Omega}\varphi(\lambda)^p\, d\mu(\lambda)\ \leq\, C^p.
$$
This implies that $\varphi(\lambda)<\infty$ for a.e. $\lambda\in\Omega$.

If $\lambda\in \Omega$ is such that $t\mapsto x_t(\lambda)$ is continuous on $(0,\infty)$, then
$$
\varphi(\lambda)=\bignorm{(x_t(\lambda))_{t>0}}_{V^q}.
$$
According to the assumption (1), this holds true almost everywhere. Hence
the quantity $\bignorm{(x_t(\lambda))_{t>0}}_{V^q}$ is finite for a.e. $\lambda\in\Omega$.
The lemma clearly follows from these properties.
\end{proof}

\medskip
\section{Variation of ergodic averages.}

Throughout we let $(\Omega,\mu)$ be a measure space and we let $1<p<\infty$.
We first recall the notion of regular operators on $L^p(\Omega)$ and some of their basic properties
which will be used in this paper. We refer e.g. to \cite[Chap. 1]{MN} and to \cite{Pe,Pi}
for more details and complements.

An operator $T\colon L^p(\Omega)\to L^p(\Omega)$  is called regular if there is a constant
$C\geq 0$ such that
$$
\bignorm{\sup_{k\geq 1}\vert T(x_k)\vert}_p\,\leq\, C\bignorm{\sup_{k\geq 1}\vert x_k\vert}_p
$$
for any finite sequence $(x_k)_{k\geq 1}$ in $L^p(\Omega)$. Then we let $\norm{T}_r$ denote the
smallest $C$ for which this holds. The set of all regular operators on $L^p(\Omega)$ is a
vector space on which $\norm{\ }_r$ is a norm.
We say that $T$ is contractively regular if $\norm{T}_r\leq 1$.

Let $E$ be a Banach space. If an operator $T\colon L^p(\Omega)\to
L^p(\Omega)$ is regular, then the operator $T\otimes I_E\colon L^p(\Omega)\otimes E\to
L^p(\Omega)\otimes E$ extends to a bounded operator on the Bochner space $L^p(\Omega; E)$, and
\begin{equation}\label{2Reg}
\bignorm{T\otimes I_E\colon L^p(\Omega;E)\longrightarrow L^p(\Omega;E)}\,\leq\,\norm{T}_r.
\end{equation}
Indeed by definition this holds true when $E=\ell^{\infty}_n$ for any $n\geq 1$, and the general
case follows from the fact that for any $\varepsilon >0$,
any finite dimensional Banach space is $(1+\varepsilon)$-isomorphic to
a subspace of $\ell^\infty_n$ for some large enough $n\geq 1$.

Any positive operator $T$ (in the lattice sense) is regular and $\norm{T}_r=\norm{T}$ in this case.
Thus all statements given for
contractively regular operators apply to positive contractions. It is well-known that conversely,
$T$ is regular with $\norm{T}_r\leq C$ if and only if there is a positive operator
$S\colon L^p(\Omega)\to L^p(\Omega)$ with $\norm{S}\leq C$, such that
$\vert T(x)\vert \leq S(\vert x\vert)$ for any $x\in L^p(\Omega)$.

Finally, following \cite{Pe}, we say that an operator $T\colon L^{1}(\Omega)+L^{\infty}(\Omega) \to
L^{1}(\Omega)+L^{\infty}(\Omega)$ is an absolute contraction if it induces two contractions
$$
T\colon L^{1}(\Omega) \longrightarrow
L^{1}(\Omega) \qquad\hbox{and}\qquad T\colon L^{\infty}(\Omega) \longrightarrow
L^{\infty}(\Omega).
$$
Then the resulting operator $T\colon L^{p}(\Omega) \to
L^{p}(\Omega)$ is contractively regular.

The main result of this section is the following theorem, which might be known to experts.
Its proof relies on the transference principle, already used in \cite{B}.

\begin{theorem}\label{2Main} Let $T\colon L^p(\Omega)\to L^p(\Omega)$ be a contractively regular operator,
with $1<p<\infty$, and let $2<q<\infty$. Then we have
$$
\bignorm{\bigl(M_n(T)x\bigr)_{n\geq 0}}_{L^p(v^q)}\,\leq\,C_{p,q}\,\norm{x}_p,\qquad x\in L^p(\Omega),
$$
for some constant $C_{p,q}$ only depending on $p$ and $q$.
\end{theorem}

Let
$$
s_p \colon \ell^p_{\footnotesize{\Zdb}} \longrightarrow \ell^p_{\footnotesize{\Zdb}},
\qquad s_p\bigl((c_n)_n\bigr)=(c_{n-1})_{n},
$$
denote the shift operator on $\ell^p_{\footnotesize{\Zdb}}$. According to \cite[Thm. B]{JKRW},
$s_p$ satisfies Theorem \ref{2Main} (the crucial case $p=2$ going back to \cite{B}). Thus
for any $1<p<\infty$ and $2<q<\infty$ we have
a constant
\begin{equation}\label{2Cpq}
C_{p,q}\,=\,\bignorm{c\mapsto \bigl(M_n(s_p)c\bigr)_{n\geq 0}}_{\ell^p_{\Zdb}\to \ell^p_{\Zdb}(v^q)}.
\end{equation}

The following lemma is a variant of the well-known Coifman-Weiss
transference Theorem \cite[Thm. 2.4]{CW} and
is closely related to \cite{DLT} and \cite{ABG}.

\begin{lemma}\label{2Transfer}
Let $U\colon L^p(\Omega)\to L^p(\Omega)$ be an invertible operator such that $U^j$ is regular
for any $j\in \Zdb$ and suppose that $C=\sup\{\norm{U^j}_r\, :\, j\in\Zdb\}\,<\,\infty\,$.
Let $(e_1,\ldots,e_N)$ be a basis of a finite dimensional Banach space $E$, and
let $a(1),\ldots, a(N)$ be $N$ elements of $\ell^1_{\footnotesize{\Zdb}}$.
Consider the two operators
$$
K\colon \ell^p_{\footnotesize{\Zdb}}\longrightarrow \ell^p_{\footnotesize{\Zdb}}(E),\qquad K(c)\,=\,\sum_{k=1}^{N}
\Bigl(\sum_{j\in\footnotesize{\Zdb}} a(k)_j\, s_p^j(c)\Bigr)\otimes e_k,
$$
and
$$
R\colon L^p(\Omega) \longrightarrow L^p(\Omega;E),\qquad R(x)\,=\,\sum_{k=1}^{N}
\Bigl(\sum_{j\in\footnotesize{\Zdb}} a(k)_j \,U^j(x)\Bigr)\otimes e_k.
$$
Then
$$
\norm{R}\leq C^2\norm{K}.
$$
\end{lemma}

\begin{proof}
The operator $I_{L^p(\Omega)}\otimes K$ extends to a bounded operator
$L^p\bigl(\Omega;\ell^p_{\footnotesize{\Zdb}}\bigr)\to L^p\bigl(\Omega;\ell^p_{\footnotesize{\Zdb}}(E)\bigr)$,
whose norm is equal to $\norm{K}$. (Nothing special about $K$ is required for this tensor extension
property.) By Fubini's Theorem,
$$
L^p\bigl(\Omega;\ell^p_{\footnotesize{\Zdb}}\bigr)\simeq \ell^p_{\footnotesize{\Zdb}}
\bigl(L^p(\Omega)\bigr)\qquad\hbox{and}\qquad
L^p\bigl(\Omega;\ell^p_{\footnotesize{\Zdb}}(E)\bigr)\simeq \ell^p_{\footnotesize{\Zdb}}\bigl(L^p(\Omega;E)\bigr)
$$
isometrically. Further,
under these identifications, the extension of
$I_{L^p(\Omega)}\otimes K$ corresponds to the  operator
$$
\widetilde{K}\colon \ell^p_{\footnotesize{\Zdb}}\bigl(L^p(\Omega)\bigr)\longrightarrow \ell^p_{\footnotesize{\Zdb}}\bigl(L^p(\Omega; E)\bigr),\qquad \widetilde{K}(z)\,=\,\sum_{k=1}^{N}
\Bigl(\sum_{j} a(k)_j\, \bigl(s_p^j\overline{\otimes} I_{L^p(\Omega)}\bigr)(z)\Bigr)\otimes e_k
$$
Thus we have $\norm{\widetilde{K}}=\norm{K}$.

By approximation, we may suppose that $a(1),\ldots,a(N)$ are finitely supported. Let $m\geq 1$
be chosen such that $a(k)_j=0$ for any $k$ and any $\vert j\vert >m$.
Let $x\in L^p(\Omega)$ and let
$$
y_k = \,\sum_{j=-m}^{m} a(k)_j\, U^j(x)
$$
for any $k=1,\ldots,N$. Our aim is to estimate the norm of $\sum_k y_k\otimes e_k\,$
in $L^p(\Omega;E)$.
For any $i\in\Zdb$, we have
$$
\sum_{k=1}^{N} y_k\otimes e_k\, =\,(U^{i}\otimes I_E)\Bigl(
\sum_{k=1}^{N} U^{-i}y_k\otimes e_k\Bigr).
$$
Hence applying (\ref{2Reg}) to $U^{i}$, we derive that
$$
\Bignorm{\sum_{k=1}^{N} y_k\otimes e_k}_{L^p(\Omega;E)}\, \leq\, C\,\Bignorm{
\sum_{k=1}^{N} U^{-i}y_k\otimes e_k}_{L^p(\Omega;E)}
$$
Let $n\geq 1$ be an arbitrary integer and let $\chi$ be the characteristic function
of the interval $[-(n+m),(n+m)]$. We deduce from the above estimate that
\begin{align*}
(2n+1)\Bignorm{\sum_{k=1}^{N} y_k\otimes e_k}_{L^p(\Omega;E)}^{p}\, & \leq\, C^p
\,\sum_{i=-n}^{n}
\Bignorm{
\sum_{k=1}^{N} \sum_{j=-m}^{m} a(k)_j\,U^{j-i}(x)\otimes e_k}_{L^p(\Omega;E)}^{p}
\\
&
\leq\, C^p
\,\sum_{i}
\Bignorm{
\sum_{k=1}^{N} \sum_{j} a(k)_j\,\chi(j-i)\,U^{j-i}(x)\otimes e_k}_{L^p(\Omega;E)}^{p}.
\end{align*}
Since
$$
\sum_i \Bignorm{
\sum_{k=1}^{N} \sum_{j} a(k)_j\,\chi(j-i)\,U^{j-i}(x)\otimes e_k}_{L^p(\Omega;E)}^{p}
\,=\,\Bignorm{\widetilde{K}\Bigl[\bigl(\chi(-j)U^{-j}
(x)\bigr)_j\Bigr]}^p_{\ell^p_{\Zdb}(L^p(\Omega; E))}
$$
and $\norm{\widetilde{K}}=\norm{K}$, this yields
\begin{align*}
(2n+1)\Bignorm{\sum_{k=1}^{N} y_k\otimes e_k}_{L^p(\Omega;E)}^{p}\,  &\leq\, C^p
\norm{K}^p \sum_{j=-(n+m)}^{n+m}\norm{U^{-j}(x)}^p_p\\
& \leq\, \bigl(2(n+m)+1\bigr)\, C^{2p} \norm{K}^p \norm{x}^p_p.
\end{align*}
Letting $n\to\infty$, we get the result.
\end{proof}

\begin{proof}[Proof of Theorem \ref{2Main}]
Let $T\colon L^p(\Omega)\to L^p(\Omega)$ be a contractively regular operator.
There exists another measure space $(\widehat{\Omega},\widehat{\mu})$,
two positive contractions $J\colon L^p(\Omega)\to L^p(\widehat{\Omega})$
and $Q\colon L^p(\widehat{\Omega})\to L^p(\Omega)$ and an isometric invertible
operator $U\colon L^p(\widehat{\Omega})\to L^p(\widehat{\Omega})$ such that
$$
T^k=QU^kJ,\qquad k\geq 0.
$$
In the case when $T$ is positive, this is  Akcoglu's famous dilation Theorem (see \cite{A}).
The extension to regular operators stated here is from \cite{Pe} or \cite{CRW}. Moreover
$U$ can be chosen so that $U$ and $U^{-1}$ are both contractively regular.
Thus
$$
\forall\, j\in\Zdb,\qquad \norm{U^j}_r=1.
$$
Note that when $1<p\not= 2<\infty$, any isometry on $L^p$ is contractively regular,
so the latter information is relevant only when $p=2$.

We fix an integer $m\geq 1$ and we consider $v^q_m$ as defined in Section 2.
For any $n\geq 0$, we clearly have 
$$
M_n(T)=QM_n(U)J.
$$
Since $\norm{Q}_r\leq 1$, it follows from
(\ref{2Reg}) that
\begin{equation}\label{2Dilation}
\bignorm{\bigl(M_n(T)x\bigr)_{0\leq n\leq m}}_{L^p(\Omega;v^q_m)}\,\leq\,
\bignorm{\bigl(M_n(U)J(x)\bigr)_{0\leq n\leq m}}_{L^p(\widehat{\Omega};v^q_m)},\qquad x\in L^p(\Omega).
\end{equation}

For any $n=0,1,\ldots, m$, let $a(n)\in \ell^1_{\footnotesize{\Zdb}}$ be defined
by letting $a(n)_j=(n+1)^{-1}$ if $0\leq j\leq n$ and $a(n)_j=0$ otherwise. Then
$$
\sum_{j\in\footnotesize{\Zdb}}a(n)_j \, U^j\, =M_n(U) \qquad
\hbox{and}\qquad
\sum_{j\in\footnotesize{\Zdb}}a(n)_j \, s_p^j\, =M_n(s_p).
$$
Applying Lemma \ref{2Transfer} with $E=v^q_m$ and recalling (\ref{2Cpq}), we therefore deduce that
$$
\bignorm{\bigl(M_n(U)z\bigr)_{0\leq n\leq m}}_{L^p(\widehat{\Omega};v^q_m)}\,\leq\, C_{p,q}\,\norm{z}_p
$$
for any $z\in L^p(\widehat{\Omega})$. Combining with the inequality (\ref{2Dilation})
we obtain that
$$
\bignorm{\bigl(M_n(T)x\bigr)_{0\leq n\leq m}}_{L^p(\Omega;v^q_m)}\,\leq\, C_{p,q}\,\norm{x}_p
$$
for any $x\in L^p(\Omega)$. Then the result follows from Lemma \ref{6Lem1}.
\end{proof}

\begin{remark}\label{2Amenable}
The above Lemma \ref{2Transfer} extends without any difficulty to amenable groups, as follows.
Let $G$ be a locally compact amenable group, with left Haar measure $dt$,
let $\pi\colon G\to B(L^p(\Omega))$ be a strongly continuous
representation valued in the space of regular operators on $L^p(\Omega)$,
and assume that
$$
C=\sup\{\norm{\pi(t)}_r\, :\, t\in G\}\,<\infty\,.
$$
Next let $h_1,\ldots, h_N$ be $N$ elements of $L^1(G)$ and let
$K\colon L^p(G)\to L^p(G;E)$ be defined by letting $K(f)=\sum_k (h_k*f)\otimes e_k\,$
for any $f\in L^p(G)$.
Then for any $x\in L^p(\Omega)$, we have
$$
\Bignorm{\sum_k\Bigl(\int_G h_k(t)\pi(t)x\, dt\,\Bigr)\otimes e_k}_{L^p(\Omega;E)}\,\leq
\,C^2\,\norm{K}\,\norm{x}_p.
$$
\end{remark}

\bigskip
We conclude this section with a continuous version of Theorem \ref{2Main}. Given a
strongly continuous semigroup $T=(T_t)_{t\geq 0}$ on $L^p(\Omega)$,
we let
$$
M_t(T)=\,\frac{1}{t}\,\int_{0}^{t} T_s\, ds\,,\qquad t>0,
$$
defined in the strong sense.

\begin{corollary}\label{2Continuous}
Let $T=(T_t)_{t\geq 0}$ be a strongly continuous semigroup on $L^p(\Omega)$ and assume that
$T_t\colon L^p(\Omega)\to L^p(\Omega)$ is contractively regular for any $t\geq 0$.
Let $2<q<\infty$ and let $x\in L^p(\Omega)$. Then for a.e. $\lambda\in\Omega$, the family
$\bigl([M_t(T)x](\lambda)\bigr)_{t>0}$ belongs to $V^q$ and
$$
\Bignorm{\lambda\mapsto \bignorm{\bigl([M_t(T)x](\lambda)\bigr)_{t> 0}}_{V^q}}_{p}\,\leq\,
C_{p,q}\,\norm{x}_p.
$$
\end{corollary}

\begin{proof}
Consider $x\in L^p(\Omega)$. According to \cite[Section VIII.7]{DS}, the function $t\mapsto [M_t(T)x](\lambda)$
is continuous for a.e. $\lambda\in\Omega$.
Let $t_0< t_1< \cdots < t_m$ be positive real numbers and let $\varepsilon >0$.
It follows from the strong continuity of $T=(T_t)_{t\geq 0}$ that there exist $\alpha>0$ and integers
$n_0, n_1,\ldots, n_m$ such that
$$
\forall\, k=0,\ldots,m,\qquad\bignorm{M_{t_k}(T)x\, - M_{n_k}(T_\alpha)x}_p\,<\varepsilon.
$$
Hence applying Theorem \ref{2Main} and a limit argument, we
deduce that
\begin{equation}\label{2Key}
\bignorm{\bigl(M_{t_0}(T)x,M_{t_1}(T)x,\ldots,M_{t_m}(T)x \bigr)}_{L^p(\Omega; v^q_m)}\,\leq C_{pq}\norm{x}_p.
\end{equation}
The result therefore follows from Lemma \ref{6Lem2}.
\end{proof}

An alternative proof of (\ref{2Key}) consists in using Fendler's dilation Theorem for semigroups (see \cite{Fe}),
and then arguing as in the proof Theorem \ref{2Main}. This only
requires knowing that the result of Corollary \ref{2Continuous} holds true for
the translation group on $L^p(\Rdb)$, which follows from \cite{B, JKRW, JR},
and using Remark \ref{2Amenable} for $G=\Rdb$ to transfer that result to strongly continuous groups
of contractively regular isometries.

\medskip
\section{The analytic case.}
  
Let $X$ be an arbitrary Banach space, and let
$(T_t)_{t\geq 0}$ be a strongly continuous semigroup on $X$.
We call it a bounded analytic semigroup if there exists a positive angle
$\omega\in\bigl(0,\frac{\pi}{2}\bigr)$ and a
bounded analytic family $z\in\Sigma_\omega\mapsto
T_z\in B(X)$ extending $(T_t)_{t>0}$,
where
$$
\Sigma_{\omega}\,=\,\bigl\{z\in \Cdb^*\, :\, \vert{\rm Arg}(z)\vert<\omega\bigr\}
$$
is the open sector of angle $2\omega$ around $(0,\infty)$. We refer to \cite{Go,Pa} for various
characterizations and properties of bounded analytic semigroups. We simply recall that
if we let $A$ denote the infinitesimal generator of $(T_t)_{t\geq 0}$, then the latter is
a bounded analytic semigroup if and only if
$T_t$ maps $X$ into the domain of $A$ for any $t>0$ and 
there exist two constants $C_0,C_1>0$ such that
\begin{equation}\label{3Anal1}
\forall\, t>0,\qquad \norm{T_t}\leq C_0
\qquad\hbox{and}\qquad
\norm{tAT_t}\leq C_1.
\end{equation}

The definition of analyticity for discrete semigroups parallels (\ref{3Anal1}).
Let $T\in B(X)$. We say that $T$ is power bounded if
$$
\sup_{n\geq 0}\norm{T^n}\,<\infty
$$
and  that it is analytic if moreover,
$$
\sup_{n\geq 1} n\norm{T^{n}-T^{n-1}}\,<\infty.
$$
This notion goes back to \cite{CSC} and has been studied in various contexts
so far. We gather here a few spectral properties
of these operators and refer to \cite{Bl2, Ly,NZ,N} for proofs and complements.
The most important result is the following:
an operator $T$ is power bounded and analytic if and only if
\begin{equation}\label{3Ritt}
\sigma(T)\subset \overline{\Ddb}\qquad\hbox{and}\qquad
\bigl\{(z-1)(zI-T)^{-1}\,:\,\vert z \vert>1\bigr\}\ \hbox{is bounded}.
\end{equation}
This property is called the `Ritt condition'.
The key argument for this characterization is due to O. Nevanlinna \cite{N}.

For any angle $\gamma\in\bigl(0,\frac{\pi}{2}\bigr)$, let
$B_\gamma$ be the convex hull of $1$ and the closed disc $\overline{D}(0,\sin\gamma)$.

\begin{figure}[ht]
\vspace*{2ex}
\begin{center}
\includegraphics[scale=0.4]{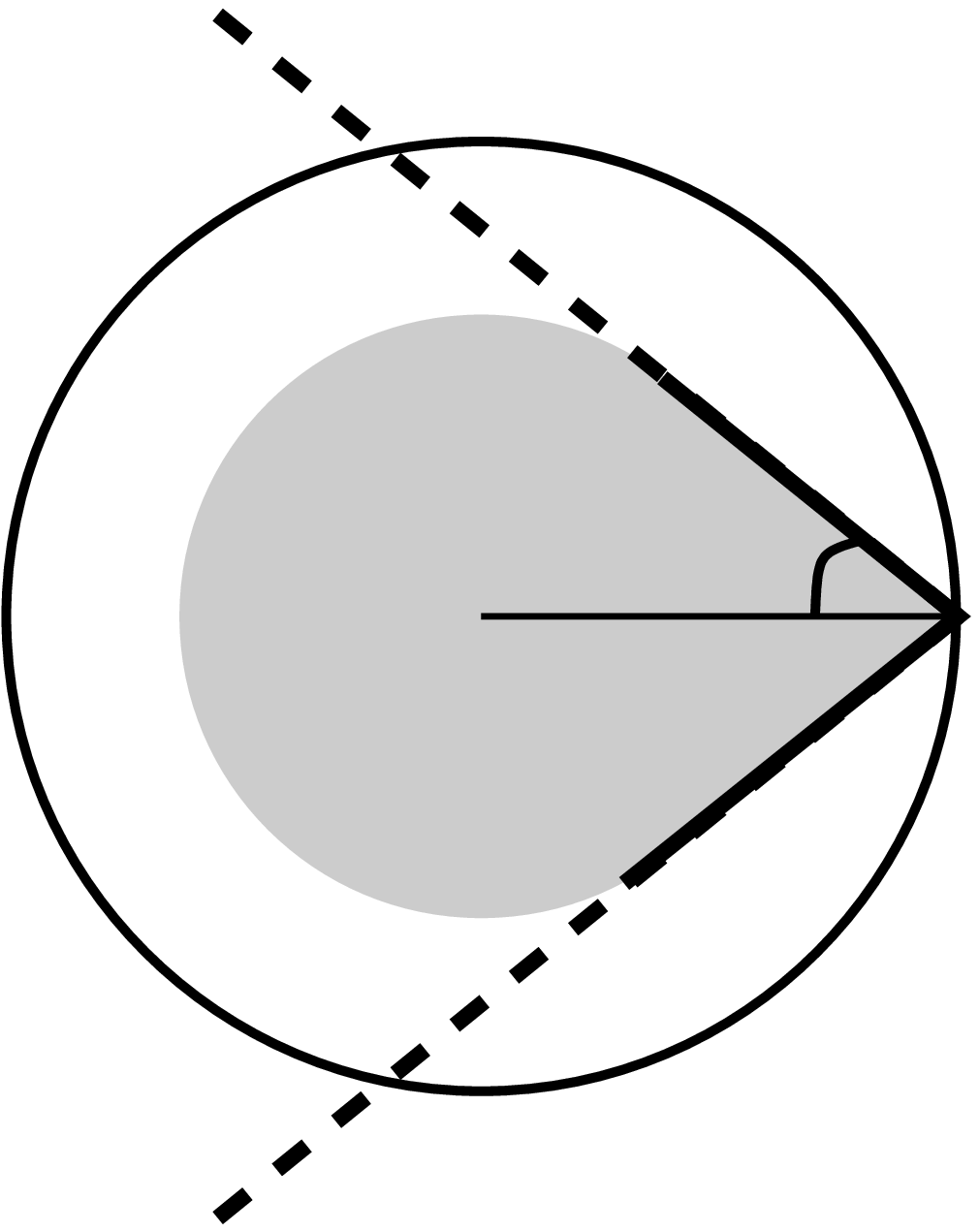}
\begin{picture}(0,0)
\put(-2,65){{\footnotesize $1$}}
\put(-68,65){{\footnotesize $0$}}
\put(-32,81){{\footnotesize $\gamma$}}
\put(-55,85){{\small $B_\gamma$}}
\end{picture}
\end{center}
\caption{\label{f1}}
\end{figure}

\noindent
Then (\ref{3Ritt}) implies that
\begin{equation}\label{3Stolz1}
\exists\,\gamma\in\bigl(0,\frac{\pi}{2}\bigr)\ \big\vert\qquad
\sigma(T)\subset B_\gamma.
\end{equation}
Furthermore, (\ref{3Stolz1}) is equivalent to
\begin{equation}\label{3Stolz2}
\exists\,K>0\ \big\vert\ \ \forall\, z\in\sigma(T),\qquad \vert 1-z\vert\leq K\bigl(1-\vert z\vert\bigr).
\end{equation}

The aim of this section is to show that under an analyticity assumption, the ergodic averages
can be replaced by the semigroup itself in either Theorem \ref{2Main} (discrete case) or
Corollary \ref{2Continuous}  (continuous case).

\bigskip
As in Section 3, we consider a measure space $(\Omega,\mu)$  and
a number $1<p<\infty$. We will consider an operator
$T\colon L^p(\Omega)\to L^p(\Omega)$ and we let
$$
\Delta_n^m= T^n(T-I)^m
$$
for any integers $n,m\geq 0$. Note that $(\Delta_n^m)_{n\geq 0}$ is the $m$-difference
sequence of $(T^n)_{n\geq 0}$.

We will need the following Littlewood-Paley type
inequalities which were estabished in \cite{LMX}.

\begin{proposition}\label{3SFE}
Let $T\colon L^p(\Omega)\to L^p(\Omega)$ be a contractively regular
operator, with $1<p<\infty$, and assume that $T$ is analytic. 
Then for any integer $m\geq 0$, there is a constant $C_m>0$ such that
$$
\Bignorm{\Bigl(\sum_{n=0}^{\infty} (n+1)^{2m+1}\bigl\vert \Delta_{n}^{m+1}(x)
\bigr\vert^2\Bigr)^{\frac{1}{2}}}_p\,\leq\,C_m \,\norm{x}_p,\qquad x\in L^p(\Omega).
$$
\end{proposition}

We will also use the following elementary estimates, whose proofs are left to the reader.

\begin{lemma}\label{3Sum}
For any integer $m\geq 0$, there exists a constant $K_m$ such that for any $n\geq 1$,
$$
\Bigl(\sum_{j=n}^{2n} (j+1)^{1-2m}\Bigr)^{\frac{1}{2}}\,\leq\, K_m n^{-m+1}.
$$
\end{lemma}

\begin{lemma}\label{3Mult}
For any sequences $(\delta_n)_{n\geq 0}\in v^1$ and $(z_n)_{n\geq 0}\in L^p(\Omega;v^q)$, we have
$(\delta_n z_n)_{n\geq 0}\in L^p(\Omega;v^q)$ and
$$
\bignorm{(\delta_n z_n)_{n\geq 0}}_{L^p(v^q)}\,\leq 3 \,
\bignorm{(\delta_n)_{n\geq 0}}_{v^1}\,\bignorm{(z_n)_{n\geq 0}}_{L^p(v^q)}.
$$
\end{lemma}

\bigskip
In the next statements and their proofs, $\lesssim$ will stand for an inequality up to
a constant which may depend on $T$, $q$ and $m$, but not on $x$.

\begin{theorem}\label{3Main} Let $T\colon L^p(\Omega)\to L^p(\Omega)$ be a contractively regular
operator, with $1<p<\infty$, and assume that $T$ is analytic. Then for any $2<q<\infty$, we have an estimate
\begin{equation}\label{3Tn}
\bignorm{\bigl(T^n (x)\bigr)_{n\geq 0}}_{L^p(v^q)}\,\lesssim\,\norm{x}_p,\qquad x\in L^p(\Omega).
\end{equation}
More generally, for any integer $m\geq 0$,
we have an estimate
\begin{equation}\label{3Delta}
\bignorm{\bigl(n^m\Delta_n^m (x)\bigr)_{n\geq 1}}_{L^p(v^q)}\,\lesssim\,\norm{x}_p,\qquad x\in L^p(\Omega).
\end{equation}
\end{theorem}

\begin{proof}
It will be convenient to set
$$
\Delta_n^{-1}\,=\,nM_{n-1}(T)\,=\sum_{j=0}^{n-1} T^j \,,\qquad n\geq 1.
$$
Then for any $n<N$ and any $m\geq -1$, we have
\begin{equation}\label{31}
\Delta_{N}^{m} -\Delta_{n}^{m}\,=\,\sum_{j=n}^{N-1}\Delta_{j}^{m+1}.
\end{equation}
With the above notation, (\ref{3Delta}) holds true for $m=-1$, by Theorem \ref{2Main}.

We will proceed by induction. We fix an integer $m\geq 0$ and assume that
(\ref{3Delta}) holds true for $(m-1)$. Thus using Lemma \ref{3Mult}, we both have
\begin{equation}\label{3Induction1}
\bignorm{\bigl(n^{m-1}\Delta_{n+1}^{m-1}(x)\bigr)_{n\geq 1}}_{L^p(v^q)}\,\lesssim\,\norm{x}_p
\end{equation}
and
\begin{equation}\label{3Induction2}
\bignorm{\bigl(n^{m-1}\Delta_{2n+1}^{m-1} (x)\bigr)_{n\geq 1}}_{L^p(v^q)}\,\lesssim\,\norm{x}_p.
\end{equation}
Next for any $n\geq 1$, we write
\begin{align*}
\sum_{j=n}^{2n} (j+1)\Delta_{j}^{m+1}\, & =\,\sum_{j=n}^{2n}(j+1)(
\Delta_{j+1}^{m} - \Delta_{j}^{m})\\
& =\,\sum_{j=n+1}^{2n+1} j\Delta_{j}^{m}\, -\,\sum_{j=n}^{2n}(j+1)\Delta_{j}^{m}\\
& =\,-\sum_{j=n+1}^{2n} \Delta_{j}^{m}\, +\, (2n+1)\Delta_{2n+1}^{m}\,-\,(n+1) \Delta_{n}^{m}\\
& =\, n\Delta_{2n+1}^{m}\,+ (n+1)(\Delta_{2n+1}^{m} - \Delta_{n}^{m})\, +\, \Delta_{n+1}^{m-1}\, -\,\Delta_{2n+1}^{m-1},
\end{align*}
using (\ref{31}) in due places.
Hence
\begin{align}\label{32}
n^m\Delta_{2n+1}^{m}\,=\,n^{m-1}\sum_{j=n}^{2n} (j+1)\Delta_{j}^{m+1}& \,-\, n^{m-1}(n+1)
(\Delta_{2n+1}^{m} - \Delta_{n}^{m})\\ \notag & +\, \,n^{m-1}\Delta_{2n+1}^{m-1}\, -\, n^{m-1}\Delta_{n+1}^{m-1}.
\end{align}
This identity suggests the introduction of the following two sequences of operators. For any  $n\geq 1$, we set
$$
A_n=\,n^{m-1}\sum_{j=n}^{2n} (j+1)\Delta_{j}^{m+1}\qquad\hbox{and}\qquad
B_n=\, n^m(\Delta_{2n+1}^{m} - \Delta_{n}^{m}).
$$
Also for any $x\in L^p(\Omega)$, we set
$$
\Phi_m(x)\,=\,\Bigl(\sum_{j=1}^{\infty}(j+1)^{2m+1}\bigl\vert \Delta_{j}^{m+1}(x)\vert^2\Bigr)^{\frac{1}{2}}.
$$
According to Proposition \ref{3SFE}, this function is an element of $L^p(\Omega)$.

Let $(n_k)_{k\geq 0}$ be an increasing sequence of integers, with $n_0=1$.
For any $k\geq 1$, we set
\begin{align*}
a_k & = \left\{
    \begin{array}{ll}\displaystyle
        n_k^{m-1}\,\sum_{j=2n_{k-1}+1}^{2n_k} (j+1)\Delta_{j}^{m+1}\qquad  & \mbox{if } 2n_{k-1}\geq n_k \\
        \displaystyle
        n_k^{m-1}\,\sum_{j=n_k}^{2n_k} (j+1)\Delta_{j}^{m+1} \qquad  & \mbox{if } 2n_{k-1} < n_k,
    \end{array}
\right.\\
\smallskip
b_k  &= \left\{
    \begin{array}{ll}\displaystyle
        -n_{k-1}^{m-1}\,\sum_{j=n_{k-1}}^{n_k -1} (j+1)\Delta_{j}^{m+1} \qquad \ & \mbox{if } 2n_{k-1}\geq n_k \\
        \displaystyle
        -n_{k-1}^{m-1}\,\sum_{j=n_{k-1}}^{2n_{k -1}} (j+1)\Delta_{j}^{m+1} \qquad \  & \mbox{if } 2n_{k-1} < n_k,
    \end{array}
\right.\\
\bigskip
c_k  & = \left\{
    \begin{array}{ll}\displaystyle
        \Bigl(n_k^{m-1} -n_{k-1}^{m-1}\Bigr)\,\sum_{j=n_{k}}^{2n_{k-1}}
        (j+1)\Delta_{j}^{m+1}\  & \mbox{if } 2n_{k-1}\geq n_k \\
        \quad 0  & \mbox{if } 2n_{k-1} < n_k.
    \end{array}
\right.
\end{align*}
This yields a decomposition
\begin{equation}\label{3Decomp}
A_{n_k} -A_{n_{k-1}}\,=\,a_k+ b_k+ c_k.
\end{equation}
Let $x\in L^p(\Omega)$. If $2n_{k-1}\geq n_k$,
we have, using Cauchy-Schwarz,
\begin{align*}
\bigl\vert a_k(x)\bigr\vert\,& \leq\,n_k^{m-1}
\sum_{j=2n_{k-1}+1}^{2n_k} (j+1)\bigl\vert \Delta_{j}^{m+1}(x)\bigr\vert\\
& \leq\, n_k^{m-1}\Bigl(\sum_{j=2n_{k-1}+1}^{2n_k} (j+1)^{1-2m}\Bigr)^{\frac{1}{2}}\,
\Bigl(\sum_{j=2n_{k-1}+1}^{2n_k} (j+1)^{2m+1}\bigl\vert \Delta_{j}^{m+1}(x)\bigr\vert^2
\Bigr)^{\frac{1}{2}}\\
& \leq\, n_k^{m-1}\Bigl(\sum_{j=n_{k}}^{2n_k} (j+1)^{1-2m}\Bigr)^{\frac{1}{2}}\,
\Bigl(\sum_{j=2n_{k-1}+1}^{2n_k} (j+1)^{2m+1}\bigl\vert \Delta_{j}^{m+1}(x)\bigr\vert^2
\Bigr)^{\frac{1}{2}}.
\end{align*}
Similarly if $2n_{k-1}< n_k$, we have
\begin{align*}
\bigl\vert a_k(x)\bigr\vert\,&
\leq\, n_k^{m-1}\Bigl(\sum_{j=n_{k}}^{2n_k} (j+1)^{1-2m}\Bigr)^{\frac{1}{2}}\,
\Bigl(\sum_{j=n_{k}}^{2n_k} (j+1)^{2m+1}\bigl\vert \Delta_{j}^{m+1}(x)\bigr\vert^2
\Bigr)^{\frac{1}{2}}\\
& \leq\, n_k^{m-1}\Bigl(\sum_{j=n_{k}}^{2n_k} (j+1)^{1-2m}\Bigr)^{\frac{1}{2}}\,
\Bigl(\sum_{j=2n_{k-1}+1}^{2n_k} (j+1)^{2m+1}\bigl\vert \Delta_{j}^{m+1}(x)\bigr\vert^2
\Bigr)^{\frac{1}{2}}.
\end{align*}
Hence in both cases,
we have
$$
\bigl\vert a_k(x)\bigr\vert^2\,\leq\,
K_m^2 \sum_{j=2n_{k-1}+1}^{2n_k} (j+1)^{2m+1}\bigl\vert \Delta_{j}^{m+1}(x)\bigr\vert^2
\,,
$$
by Lemma \ref{3Sum}. Summing up, we deduce that
\begin{equation}\label{3Alpha}
\sum_{k=1}^{\infty}\bigl\vert a_k(x)\bigr\vert^2\,\leq\,
K_m^2 \,\Phi_m(x)^2.
\end{equation}
Likewise we have
\begin{equation}\label{3Beta}
\sum_{k=1}^{\infty}\bigl\vert b_k(x)\bigr\vert^2\,\leq\,
K_m^2 \,\Phi_m(x)^2.
\end{equation}
We now turn to $c_k(x)$. Assume that $2n_{k-1}\geq n_k$. Then using again
Cauchy-Schwarz and Lemma \ref{3Sum}, we have
$$
\bigl\vert c_k(x)\bigr\vert\,
\leq\, \bigl\vert n_k^{m-1} -n_{k-1}^{m-1}\bigr\vert\,\Bigl(\sum_{j=n_{k}}^{2n_{k-1}}
(j+1)^{1-2m}\Bigr)^{\frac{1}{2}}\,\Bigl(\sum_{j=n_{k}}^{2n_{k-1}}
(j+1)^{2m+1}\bigl\vert \Delta_{j}^{m+1}(x)\bigr\vert^2
\Bigr)^{\frac{1}{2}},
$$
hence
$$
\bigl\vert c_k(x)\bigr\vert^2\,\leq\, K_m^2\,\biggl(
\frac{n_k^{m-1} -n_{k-1}^{m-1}}{n_k^{m-1}}\biggr)^2\
\sum_{j=n_{k}}^{2n_{k-1}}
(j+1)^{2m+1}\bigl\vert \Delta_{j}^{m+1}(x)\bigr\vert^2.
$$
For any integer $j\geq 1$, define
$$
J_j=\{k\geq 1\, :\, n_{k}\leq j\leq 2n_{k-1}\},
$$
and set
$$
\Lambda_j\,=\,\sum_{k\in J_j}\biggl(
\frac{n_k^{m-1} -n_{k-1}^{m-1}}{n_k^{m-1}}\biggr)^2\,.
$$
Then it follows from the above calculation that
$$
\sum_{k=1}^{\infty}\bigl\vert c_k(x)\bigr\vert^2\,\leq\,
K_m^2\sum_{j=1}^{\infty} \Lambda_j\,(j+1)^{2m+1}\bigl\vert \Delta_{j}^{m+1}(x)\bigr\vert^2\,.
$$
Let us now estimate the $\Lambda_j$'s.
Observe that if $J_j$ is a non empty set, then it is a finite interval of integers. Thus it reads as
$$
J_j=\{k_j-N+1, k_j-N+2,\ldots, k_j-1, k_j\},
$$
where $k_j$ is the biggest element of $J_j$ and $N$ is its cardinal.

Suppose that $m\geq 2$, so that
the sequence $(n_k^{m-1})_k$ is increasing.
Then
$$
\sum_{k\in J_j}
n_k^{m-1} -n_{k-1}^{m-1}\, = \,\sum_{r=0}^{N-1} n_{k_j-r}^{m-1} - n_{k_j -r-1}^{m-1}\,
=  \, n_{k_j}^{m-1} - n_{k_j -N}^{m-1}\,\leq n_{k_j}^{m-1}.
$$
Since $k_j\in J_j$, we have $n_{k_j}\leq j$, hence
$$
\sum_{k\in J_j}
n_k^{m-1} -n_{k-1}^{m-1}\,\leq j^{m-1}.
$$
On the other hand,  we have $j\leq 2n_k$ for any $k\in J_j$, hence
$$
\sum_{k\in J_j}
\frac{n_k^{m-1} -n_{k-1}^{m-1}}{n_k^{m-1}}\, \leq \, \Bigl(\frac{2}{j}\Bigr)^{m-1}\,
\sum_{k\in J_j}
n_k^{m-1} -n_{k-1}^{m-1}\,
\leq \, 2^{m-1}.
$$
We immediatly deduce that
$$
\Lambda_j\leq 4^{m-1}.
$$
In the case when $m=0$, we have similarly
$$
\sum_{k\in J_j}
\frac{n_{k-1}^{-1} -n_{k}^{-1}}{n_k^{-1}}\, \leq \,j\,
\sum_{k\in J_j} n_{k-1}^{-1} -n_{k}^{-1}\\
\leq \,j\,
n_{k_j-N}^{-1}\leq 2.
$$
Hence $\Lambda_j\leq 4$ in this case. Lastly, it is plain that if $m=1$, we have $\Lambda_j=0$.

This shows that in all cases, we have an estimate
$$
\sum_{k=1}^{\infty}\bigl\vert c_k(x)\bigr\vert^2\,\leq\,
{K'_m}^2\,\Phi_m(x)^2.
$$
Now recall (\ref{3Decomp}). Combining the above estimate with (\ref{3Alpha}) and (\ref{3Beta}), we obtain
that
\begin{align*}
\sum_{k=1}^{\infty}
\bigl\vert A_{n_k}(x) -A_{n_{k-1}}(x)\bigr\vert^2\,\leq &\, 3\Bigl(
\sum_{k=1}^{\infty} \bigl\vert a_k(x)\bigr\vert^2\, +\,
\sum_{k=1}^{\infty} \bigl\vert b_k(x)\bigr\vert^2\, +\,
\sum_{k=1}^{\infty} \bigl\vert c_k(x)\bigr\vert^2\Bigr)\\
&\leq \bigl(6K_m^2+{K'_m}^2\bigr)\,\Phi_m(x)^2.
\end{align*}
Since the upper bound does not depend on the sequence $(n_k)_{k\geq 0}$, this estimate and
Proposition \ref{3SFE} imply that the sequence
$(A_n(x))_{n\geq 0}$ belongs to $L^p(\Omega;v^2)$, and that we have an estimate
\begin{equation}\label{3A}
\bignorm{\bigl(A_n (x)\bigr)_{n\geq 1}}_{L^p(v^2)}\,\lesssim\,\norm{x}_p,\qquad x\in L^p(\Omega).
\end{equation}

We will now apply a similar treatment to the sequence $(B_n)_n$. According to (\ref{31}), we can write
$$
B_n=n^m\,\sum_{j=n}^{2n} \Delta_{j}^{m+1}.
$$
For any $k\geq 1$, we set
\begin{align*}
\alpha_k & = n_k^{m}\,\sum_{j=2n_{k-1}+1}^{2n_k} \Delta_{j}^{m+1},\\
\beta_k  & = -n_{k-1}^{m}\,\sum_{j=n_{k-1}}^{n_k -1} \Delta_{j}^{m+1},\\
\gamma_k  & = \Bigl(n_k^{m} -n_{k-1}^{m}\Bigr)\,\sum_{j=n_{k}}^{2n_{k-1}}
\Delta_{j}^{m+1}
\end{align*}
if $2n_{k-1}\geq n_k$, and
\begin{align*}
\alpha_k & = n_k^{m}\,\sum_{j=n_k}^{2n_k} \Delta_{j}^{m+1},\\
\beta_k  & = -n_{k-1}^{m}\,\sum_{j=n_{k-1}}^{2n_{k -1}}  \Delta_{j}^{m+1},\\
\gamma_k  & = 0
\end{align*}
if $2n_{k-1} < n_k$.

Arguing as above, we obtain that for any $x\in L^p(\Omega)$, we have
$$
\bigl\vert\alpha_k(x)\bigr\vert\,\leq \,
n_k^{m}\Bigl(\sum_{j=n_{k}}^{2n_k} (j+1)^{-1-2m}\Bigr)^{\frac{1}{2}}\,
\Bigl(\sum_{j=2n_{k-1}+1}^{2n_k} (j+1)^{2m+1}\bigl\vert \Delta_{j}^{m+1}(x)\bigr\vert^2
\Bigr)^{\frac{1}{2}},
$$
and then
$$
\sum_{k=1}^{\infty}\bigl\vert \alpha_k(x)\bigr\vert^2\,\leq\,
K_{m+1}^2 \,\Phi_m(x)^2.
$$
Likewise we have
$$
\sum_{k=1}^{\infty}\bigl\vert \beta_k(x)\bigr\vert^2\,\leq\,
K_{m+1}^2 \,\Phi_m(x)^2,
$$
as well as an estimate
$$
\sum_{k=1}^{\infty}\bigl\vert \gamma_k(x)\bigr\vert^2\,\leq\,
K^{'2}_{m+1}\,\Phi_m(x)^2.
$$
These three inequalities imply that the sequence
$(B_n(x))_{n\geq 0}$ belongs to $L^p(\Omega;v^2)$, and that we have an estimate
\begin{equation}\label{3B}
\bignorm{\bigl(B_n (x)\bigr)_{n\geq 1}}_{L^p(v^2)}\,\lesssim\,\norm{x}_p,\qquad x\in L^p(\Omega).
\end{equation}

We can now conclude our proof. Recall that $q>2$, so that $v^2\subset v^q$.
Then using (\ref{3Induction1}), (\ref{3Induction2}), (\ref{3A}), (\ref{3B})
and Lemma \ref{3Mult}, it follows from the decomposition formula (\ref{32}) that for any $x\in L^p(\Omega)$,
$\bigl(n^m\Delta_{2n+1}^m(x)\bigr)_{n\geq 1}$ belongs to $L^p(\Omega;v^q)$ and that we have an estimate
$$
\bignorm{\bigl(n^m\Delta_{2n+1}^m(x)\bigr)_{n\geq 1}}_{L^p(v^q)}\,\lesssim\,\norm{x}_p.
$$
Since $n\Delta_{n}^{m}= n^{m}\Delta_{2n+1}^{m} - B_{n}$, a second application of (\ref{3B}) yields (\ref{3Delta}).
\end{proof}

The following is an analog of Theorem \ref{3Main} for continuous semigroups.

\begin{corollary}\label{3Continuous}
Let $(T_t)_{t\geq 0}$ be a bounded analytic semigroup on $L^p(\Omega)$
and assume that $T_t\colon L^p(\Omega)\to L^p(\Omega)$ is contractively regular for any
$t\geq 0$. Let $2<q<\infty$ and let $x\in L^p(\Omega)$.
Then for a.e. $\lambda\in\Omega$, the family
$\bigl([T_t (x)](\lambda)\bigr)_{t>0}$ belongs to $V^q$ and we have an estimate
\begin{equation}\label{3Cont1}
\Bignorm{\lambda\mapsto \bignorm{\bigl([T_t(x)](\lambda)\bigr)_{t> 0}}_{V^q}}_{p}\,\lesssim\,\norm{x}_p.
\end{equation}
More generally, for any
integer $m\geq 0$, the family
$\Bigl(t^m \frac{\partial^m}{\partial t^m}\bigl(T_t(x)\bigr)(\lambda)\Bigr)_{t>0}$
belongs to $V^q$ for a.e. $\lambda\in\Omega$ and we have an estimate
\begin{equation}\label{3Cont2}
\Bignorm{\bignorm{\lambda \mapsto \Bigl(t^m 
\frac{\partial^m}{\partial t^m}\bigl(T_t(x)\bigr)(\lambda)\Bigr)_{t>0}}_{V^q}}_p
\,\lesssim\,\norm{x}_p,\qquad x\in L^p(\Omega).
\end{equation}
\end{corollary}

\begin{proof} Let $m\geq 0$ be an integer. It follows from \cite[Lemma, p. 72]{S2} 
that for any $x\in L^p(\Omega)$, the function
$$
t\mapsto t^m \frac{\partial^m}{\partial t^m}\bigl( T_t(x)\bigr)(\lambda)
$$
is continuous for a.e. $\lambda\in \Omega$.

Next it follows from the proof of \cite[Cor. 4.2]{LMX} that
there exists a constant $C_m>0$ such that
$$
\Bignorm{\Bigl(\sum_{n=0}^{\infty} (n+1)^{2m+1}\bigl\vert T_t^n(T_t-I)^m(x)
\bigr\vert^2\Bigr)^{\frac{1}{2}}}_p\,\leq\,C_m \,\norm{x}_p
$$
for any $t>0$ and any $x\in L^p(\Omega)$. That is, the operators $T_t$ satisfy Proposition
\ref{3SFE} uniformly. Since they also satisfy Theorem \ref{2Main} uniformly,
it follows from the  proof
of Theorem \ref{3Main} that they satisfy the estimate (\ref{3Delta}) uniformly.
Hence using an approximation argument as in the proof of \cite[Cor. 4.2]{LMX}, we deduce that
for any $x\in L^p(\Omega)$ and for any $0<t_0<t_1<\cdots<t_m$, we have
$$
\bignorm{\bigl(T_{t_0}(x),T_{t_1}(x),\ldots,T_{t_m}(x) \bigr)}_{L^p(\Omega; v^q_m)}\,\leq C.
$$
The result therefore follows from Lemma \ref{6Lem2}.
\end{proof}

\medskip
\section{Additional properties.}

We give here further properties of contractively regular operators
and contractively regular semigroups, in connection with variational inequalities.
Let $(\Omega,\mu)$ be a measure space and let $1<p<\infty$.

\medskip
{\it 5.1. Individual ergodic theorems}

\smallskip
Let $T\colon L^p(\Omega)\to L^p(\Omega)$ be a contraction. According to the
Mean Ergodic Theorem, we have a direct sum decomposition
$$
L^p(\Omega)\,=\,N(I-T)\oplus\overline{R(I-T)},
$$
where $N(\cdotp)$ and $R(\cdotp)$ denote the kernel and the range, respectively.
Moreover if we let $P_T\colon L^p(\Omega)\to L^p(\Omega)$ denote the
corresponding projection onto $N(I-T)$, then
\begin{equation}\label{7Mean}
M_n(T)x\,\stackrel{L^p}{\longrightarrow}\, P_T(x)
\end{equation}
for any $x\in L^p(\Omega)$. It is well-known that if $T$ is an absolute
contraction, then $M_n(T)x\to P_T(x)$ almost everywhere (see e.g. \cite[Section VIII.6]{DS}).
We extend this classical result, as follows.

\begin{corollary}
Assume that $T\colon L^p(\Omega)\to L^p(\Omega)$ is contractively regular.
Then for any $x\in L^p(\Omega)$,
$$
[M_n(T)x](\lambda)\,\longrightarrow\, [P_T(x)](\lambda)\qquad \hbox{for a.e.}\  \lambda\in\Omega.
$$
\end{corollary}

\begin{proof} Let $x\in L^p(\Omega)$ and let $2<q<\infty$. According to Theorem \ref{2Main}, the sequence
$\bigl([M_n(T)x](\lambda)\bigr)_{n\geq 0}$ belongs to $v^q$ for almost every $\lambda\in\Omega$.
Hence $\bigl([M_n(T)x](\lambda)\bigr)_{n\geq 0}$ converges for almost every $\lambda\in\Omega$.
Combining with (\ref{7Mean}), we obtain the result.
\end{proof}

If a contraction $T\colon L^p(\Omega)\to L^p(\Omega)$ is analytic, then
$\norm{T^n-T^{n+1}}\to 0$, hence $T^n(x)\to 0$ for any $x\in R(I-T)$.
Consequently,
$$
T^n(x)\,\stackrel{L^p}{\longrightarrow}\, P_T(x)
$$
for any $x\in L^p(\Omega)$. Using Theorem \ref{3Main} and arguing as above, we obtain the following.

\begin{corollary}\label{72}
Let $T\colon L^p(\Omega)\to L^p(\Omega)$ be a contractively regular operator and assume that $T$ is analytic.
Then for any $x\in L^p(\Omega)$,
$$
[T^n(x)](\lambda)\,\longrightarrow\, [P_T(x)](\lambda)\qquad \hbox{for a.e.}\  \lambda\in\Omega.
$$
\end{corollary}

\bigskip
We now consider the continuous case. The situation is essentially similar, except that
we can also consider the behaviour when the parameter $t$ tends to $0^+$.
Let $T=(T_t)_{t\geq 0}$ be a strongly continuous semigroup
of contractions. By definition, for any $x\in L^p(\Omega)$, $T_t(x)\to x$ in the $L^p$-norm 
when $t\to 0^+$. 
This implies that $M_t(T)x\to x$ when $t\to 0^+$.
Let $A$ denote the infinitesimal generator of $T=(T_t)_{t\geq 0}$. As in the discrete case,
we have a direct sum decomposition
$$
L^p(\Omega)\,=N(A)\oplus \overline{R(A)}.
$$
Moreover if we let $P_A\colon L^p(\Omega)\to L^p(\Omega)$ denote the
corresponding projection onto $N(A)$, then
$$
M_t(T)x\,\stackrel{L^p}{\longrightarrow}\, P_A(x)\quad\hbox{when}\ t\to\infty
$$
for any $x\in L^p(\Omega)$. If further $(T_t)_{t\geq 0}$ is a bounded analytic semigroup,
then
$$
T_t(x)\,\stackrel{L^p}{\longrightarrow}\, P_A(x)
$$
for any $x\in L^p(\Omega)$.

Now applying Corollaries \ref{2Continuous} and \ref{3Continuous}, we deduce
the following individual ergodic theorems.

\begin{corollary}\label{74}
Let $T=(T_t)_{t\geq 0}$ be a strongly continuous semigroup of contractively
regular operators on $L^p(\Omega)$, and
let $x\in L^p(\Omega)$. Then for almost every $\lambda\in\Omega$,
$$
[M_t(T)x](\lambda)\,\longrightarrow\, [P_A(x)](\lambda)\quad\hbox{when}\ t \to\infty
$$
and
$$
[M_t(T)x](\lambda)\,\longrightarrow\, x(\lambda)\quad\hbox{when}\ t\, \to 0^+.
$$
\end{corollary}

\begin{corollary}\label{74}
Let $T=(T_t)_{t\geq 0}$ be a bounded analytic semigroup and
assume that $T_t$ is  contractively
regular for any $t\geq 0$. Let $x\in L^p(\Omega)$.
Then for almost every $\lambda\in\Omega$,
$$
[T_t(x)](\lambda)\,\longrightarrow\, [P_A(x)](\lambda)\quad\hbox{when}\ t\to\infty
$$
and
$$
[T_t(x)](\lambda)\,\longrightarrow\, x(\lambda)\quad\hbox{when}\ t\to 0^+.
$$
\end{corollary}

\bigskip
{\it 5.2. The case $q=2$}

\smallskip
In this section,
we fix a increasing sequence $(n_k)_{k\geq 0}$ of integers, with $n_0=0$. Given any sequence
$(a_n)_{n\geq 0}$ of complex numbers, we define the so-called oscillation norm
$$
\norm{(a_n)_{n\geq 0}}_{o^2}\,=\, \bigl(\vert a_0\vert^2\, +\,\sum_{k\geq 0}\,\max_{n_{k}\leq n,m\leq n_{k+1}}
\vert a_{n}-a_{m}\vert^2 \bigr)^{\frac{1}{2}},
$$
and we let $o^2$ denote the Banach space of all sequences with a finite oscillation norm, equipped with
$\norm{\ }_{o^2}$. This space (whose definition depends on the sequence $(n_k)_{k\geq 0}$) was used
in \cite{B,JKRW,JR} as a substitute to $v^q$ in the case $q=2$ (see also \cite{Ga}). Indeed,
neither Theorem \ref{2Main} nor Theorem \ref{3Main} holds true for $q=2$, see \cite{JKRW}
and  \cite[Section 8]{JW}.

Recall the shift operator $s_p\colon \ell^p_{\footnotesize{\Zdb}}\to \ell^p_{\footnotesize{\Zdb}}$
for any $1<p<\infty$ (see Section 3).
According to \cite[Thm. A]{JKRW}, there is a constant $C_{p,2}$ such that
$$
\bignorm{\bigl(M_n(s_p)c\bigr)_{n\geq 0}}_{L^p(o^2)}\,\leq\,C_{p,2}\,\norm{c}_p
$$
for any $c\in \ell^p_{\footnotesize{\Zdb}}$.
Hence arguing as in the proof of Theorem \ref{2Main}, we obtain the following $o^2$-version of the latter
statement.

\begin{theorem}\label{41} Let $T\colon L^p(\Omega)\to L^p(\Omega)$ be a contractively regular operator,
with $1<p<\infty$. Then we have
$$
\bignorm{\bigl(M_n(T)x\bigr)_{n\geq 0}}_{L^p(o^2)}\,\leq\,C_{p,2}\,\norm{x}_p,\qquad x\in L^p(\Omega).
$$
\end{theorem}

We also have an $o^2$-version of Theorem \ref{3Main}, as follows.

\begin{theorem}\label{42}
Let $T\colon L^p(\Omega)\to L^p(\Omega)$ be a contractively regular
operator, with $1<p<\infty$, and assume that $T$ is analytic. Then we have an estimate
\begin{equation}\label{4Tn}
\bignorm{\bigl(T^n (x)\bigr)_{n\geq 0}}_{L^p(o^2)}\,\lesssim\,\norm{x}_p,\qquad x\in L^p(\Omega).
\end{equation}
More generally, for any integer $m\geq 0$,
we have an estimate
\begin{equation}\label{4Delta}
\bignorm{\bigl(n^m\Delta_n^m (x)\bigr)_{n\geq 1}}_{L^p(o^2)}\,\lesssim\,\norm{x}_p,\qquad x\in L^p(\Omega).
\end{equation}
\end{theorem}

\begin{proof}
The proof is a variant of the one written for Theorem \ref{3Main}, let us explain this briefly.
We use the notation from Section 4. For any $m\geq -1$, consider the following three properties:

\begin{itemize}
\item [(i)$_{m}$] $\ \bignorm{\bigl(n^{m}\Delta_n^{m}(x)\bigr)_{n\geq 1}}_{L^p(o^2)}\,\lesssim\,\norm{x}_p$;

\smallskip
\item [(ii)$_{m}$] $\ \bignorm{\bigl(n^{m}\Delta_{n+1}^{m} (x)\bigr)_{n\geq 1}}_{L^p(o^2)}\,\lesssim\,\norm{x}_p$;

\smallskip
\item [(iii)$_{m}$] $\ \bignorm{\bigl(n^{m}\Delta_{2n+1}^{m} (x)\bigr)_{n\geq 1}}_{L^p(o^2)}\,\lesssim\,\norm{x}_p$.
\end{itemize}
Property (i)$_{m}$ is the result we wish to prove. Our strategy is to
show, by induction, that these three estimates hold true.

First, it is easy to deduce from Proposition \ref{3SFE} that  (i)$_{m}$ and (ii)$_{m}$ are equivalent.
Second  it follows from (\ref{3A}), (\ref{3B}) and (\ref{32}) that
(ii)$_{m-1}$ and (iii)$_{m-1}$ imply (iii)$_{m}$ and (i)$_{m}$. Indeed $v^2\subset o^2$
and $n^m\Delta_n^m = n^m\Delta_{2n+1}^m -B_n$. Hence it suffices to show 
(i)$_{-1}$ and (iii)$_{-1}$. This is obtained by applying Theorem \ref{41} twice, the first time for
the $o^2$-space associated with the sequence $(n_k)_{k\geq 1}$, the second time for 
$o^2$-space associated with the sequence $(2n_k+1)_{k\geq 1}$.
\end{proof}

There are also $o^2$-versions of Corollary \ref{2Continuous} and
Corollary \ref{3Continuous}, whose statements are left to the reader.

\bigskip
{\it 5.3. Jump functions}

\smallskip

It is well known that variational inequalities for a sequence of operators have consequences
in terms of jump functions. For any $\tau>0$ and any sequence
$a= (a_n)_{n\geq 0}$ of complex numbers, let $N(a,\tau)$ denote
the number of $\tau$-jumps of $a$, defined as
the supremum of all integers $N\geq 0$ for which there exist integers
$$
0\leq n_1 <m_1\leq n_2 < m_2\leq \cdots\leq n_N < m_N,
$$
such that $\vert a_{m_k}-a_{n_k}\vert >\tau$ for each $k=1,\ldots,N$.
It is clear that for any $1\leq q<\infty$,
$$
\tau^q N(a,\tau)\,\leq\,\norm{a}_{v^q}^q.
$$
Combining with Theorem \ref{2Main} and Theorem \ref{3Main}, we immediately
obtain (as in \cite[Thm. 3.15]{JR}) the following jump estimates.

\begin{corollary}\label{3Jumps}
Consider $1<p<\infty$ and $2<q<\infty$. Let $T\colon L^p(\Omega,\mu)\to L^p(\Omega,\mu)$ be a contractively
regular operator.
\begin{itemize}
\item [(1)] We have an estimate
$$
\Bignorm{\lambda\mapsto N\Bigl(\bigl([M_n(T)x](\lambda)\bigr)_{n\geq 0}\, ,\, \tau\Bigr)^{\frac{1}{q}}}_p\,\lesssim \, \frac{\norm{x}_p}{\tau}\,,
$$
and for any $K>0$, we also have
$$
\mu\Bigl\{\lambda\in\Omega\,\Big\vert\,
N\Bigl(\bigl([M_n(T)x](\lambda)\bigr)_{n\geq 0}\, ,\, \tau\Bigr)\,>\,
K\Bigr\}\,\lesssim\,
\frac{\norm{x}_p^p}{\tau^p K^{\frac{p}{q}}}\,,
$$
\item [(2)] Assume moreover that $T$ is analytic. Then we have similar estimates
$$
\Bignorm{\lambda\mapsto N\Bigl(\bigl([T^n (x)](\lambda)\bigr)_{n\geq 0}
\, ,\, \tau\Bigr)^{\frac{1}{q}}}_p\,\lesssim \, \frac{\norm{x}_p}{\tau}\,,
$$
and
$$
\mu\Bigl\{\lambda\in\Omega\,\Big\vert\,
N\Bigl(\bigl([ T^n (x)](\lambda)\bigr)_{n\geq 0}\, ,\, \tau\Bigr)\,>\,
K\Bigr\}\,\lesssim\,
\frac{\norm{x}_p^p}{\tau^p K^{\frac{p}{q}}}\,,
$$
Furthermore for any integer $m\geq 0$, similar results hold with
$n^m\Delta_n^m$ instead of $T^n$.
\end{itemize}
\end{corollary}

Similar results for the continuous case can be deduced as well from Corollary \ref{2Continuous} and
Corollary \ref{3Continuous}.

\medskip
\section{Examples and applications.}

We will now exhibit various classes of operators or semigroups
to which our results from Section 4 and Section 5 apply. We 
focus on statements involving $q$-variation, although statements involving
the oscillation norm from the subsection 5.2 are also possible.

We start with the continuous case.
Let $(T_t)_{t\geq 0}$ be a strongly continuous semigroup on $L^2(\Omega)$ and assume that each
$T_t$ is an absolute contraction, that is,
\begin{equation}\label{3Diff1}
\norm{T_t(x)}_1\leq\norm{x}_1\qquad\hbox{and}\qquad
\norm{T_t(x)}_\infty\leq\norm{x}_\infty
\end{equation}
for any $x\in L^1(\Omega)+L^\infty(\Omega)$ and any $t>0$
(see Section 3). Thus for any $1<p<\infty$,
$(T_t)_{t\geq 0}$ extends to a strongly continuous semigroup of contractively regular operators.
It is well-known that $(T_t)_{t\geq 0}$ is a bounded analytic semigroup on $L^p(\Omega)$
for every $1<p<\infty$ if (and only if) it is a bounded analytic semigroup on $L^p(\Omega)$
for one $1<p<\infty$. Applying Corollary \ref{3Continuous} and Corollary \ref{74}, we derive the following.

\begin{corollary}\label{3Diff2}
Let  $(T_t)_{t\geq 0}$ be a bounded analytic semigroup on $L^2(\Omega)$ satisfying
(\ref{3Diff1}). Then it
satisfies estimates (\ref{3Cont1}) and (\ref{3Cont2}) for every $1<p<\infty$ and every $2<q<\infty$.
Moreover for any $x\in L^p(\Omega)$, $T_t(x)$ converges almost everywhere when $t\to 0^+$ and when 
$t\to\infty$.
\end{corollary}

Note that if each $T_t\colon L^2(\Omega)\to L^p(\Omega)$ is selfadjoint, 
then $(T_t)_{t\geq 0}$ is a bounded analytic semigroup on $L^2(\Omega)$. Hence 
the above corollary applies to symmetric diffusion semigroups and extends \cite[Thm. 3.3]{JR}.

We now consider the so-called subordinated semigroups. Let $1<p<\infty$ and let
$(T_t)_{t\geq 0}$ be a strongly continuous bounded semigroup on $L^p(\Omega)$. Let
$A$ denote its infinitesimal generator and let $0<\alpha<1$. Then $-(-A)^\alpha$ generates a
bounded analytic semigroup $(T_{\alpha,t})_{t\geq 0}$ on $L^p(\Omega)$, and for any $t>0$,
there exists a continuous function $f_{\alpha,t}\colon(0,\infty)\to\Rdb\,$ such that
\begin{equation}\label{3Subord1}
\forall\, s>0,\quad f_{\alpha,t}(s)\geq 0; \qquad \int_{0}^{\infty} f_{\alpha,t}(s)\, ds\, =1;
\end{equation}
and
\begin{equation}\label{3Subord2}
T_{\alpha,t}(x)\,=\, \int_{0}^{\infty} f_{\alpha,t}(s)\, T_s(x)\, ds \, ,\qquad
x\in L^p(\Omega).
\end{equation}
See e.g. \cite[IX.11]{Y} for details and complements.

\begin{corollary}\label{3Subord}
Let $(T_t)_{t\geq 0}$ be a strongly continuous semigroup on $L^p(\Omega)$, with $1<p<\infty$,
and assume that $T_t\colon L^p(\Omega)\to L^p(\Omega)$ is contractively regular for any
$t\geq 0$. Then for any $0<\alpha<\infty$ and any $2<q<\infty$, we have an estimate
$$
\Bignorm{\lambda\mapsto \bignorm{\bigl([T_{\alpha,t}(x)](\lambda)\bigr)_{t> 0}}_{V^q}}_{p}\,\lesssim\,\norm{x}_p
$$
for the subordinated semigroup $T_{\alpha,t}=e^{-t(-A)^\alpha}$. Moreover for any $x\in L^p(\Omega)$,
$T_{\alpha,t}(x)$ converges almost everywhere when $t\to 0^+$ and when
$t\to\infty$.
\end{corollary}

\begin{proof}
Let $0<\alpha<\infty$. It follows from (\ref{3Subord1}) and (\ref{3Subord2}) that
for any $t>0$,
$$
\norm{T_{\alpha,t}}_r\leq \int_{0}^{\infty} f_{\alpha,t}(s)\,\norm{T_s}_r\, ds\,\leq 1.
$$
Hence the result follows from Corollary \ref{3Continuous} and Corollary \ref{74}.
\end{proof}

We now turn to the discrete case and consider normal operators on $L^2$.

\begin{lemma}\label{3Normal1}
Let $H$ be a Hilbert space and let $T\in B(H)$ be a normal operator. Then
$T$ is an analytic power bounded operator if and only if it satisfies
(\ref{3Stolz2}).
\end{lemma}

\begin{proof} The `only if' part holds for any operator, as discussed at the beginning of Section 4.
Conversely, assume that a normal operator $T\colon H\to H$ satisfies
(\ref{3Stolz2}).
Applying the Spectral Theorem, we deduce that $T$ is a contraction and that
for some  constant $K>0$, we have
\begin{align*}
n\norm{T^n-T^{n-1}}\, & =\,\sup_{z\in\sigma(T)} n\vert z^n -z^{n-1}\vert
\, =\,\sup_{z\in\sigma(T)} n\vert z\vert^{n-1}\,\vert 1-z\vert\,\\ & \leq\,
K\sup_{r\in[0,1]} n(r^{n-1} -r^n)\,=\,K \Bigl(\frac{n-1}{n}\Bigr)^{n-1}
\end{align*}
for any $n\geq 1$. Hence the sequence $\bigl(n(T^n-T^{n-1})\bigr)_{n\geq 1}$ is bounded.
\end{proof}

The following is a straightforward consequence of the above lemma, Theorem \ref{3Main} and Corollary \ref{72}.

\begin{corollary}\label{3Normal2} Let
$T\colon L^2(\Omega)\to L^2(\Omega)$ be a contractively regular normal operator satisfying
(\ref{3Stolz2}). Then it satisfies an estimate
$$
\bignorm{\bigl(T^n (x)\bigr)_{n\geq 0}}_{L^2(v^q)}\,\lesssim\,\norm{x}_2,\qquad x\in L^2(\Omega)
$$
for any $2<q<\infty$. Moreover for any $x\in L^2(\Omega)$, $T^n(x)$ converges almost everywhere 
when $n\to\infty$.
\end{corollary}

The following is a discrete analog of Corollary \ref{3Diff2}.

\begin{corollary}\label{3Diff3}
Let $T\colon L^1(\Omega)+L^\infty(\Omega)\to L^1(\Omega)+L^\infty(\Omega)$ be an absolute
contraction and assume that $T\colon L^2(\Omega)\to L^2(\Omega)$ is analytic. Then $T$
satisfies estimates (\ref{3Tn}) and (\ref{3Delta}) for every $1<p<\infty$ and every $2<q<\infty$.
Moreover for any $x\in L^p(\Omega)$, $T^n(x)$ converges almost everywhere
when $n\to\infty$.
\end{corollary}

\begin{proof} If $T\colon L^2(\Omega)\to L^2(\Omega)$ is analytic, then
$T\colon L^p(\Omega)\to L^p(\Omega)$ is analytic as well for any $1<p<\infty$, by
\cite[Thm. 1.1]{Bl2}. Hence the result follows from Theorem \ref{3Main} and Corollary \ref{72}.
\end{proof}

Of course the above corollary applies if $T\colon L^2(\Omega)\to L^2(\Omega)$ is a
positive selfadjoint operator, more generally if $T$
is normal and satisfies (\ref{3Stolz2}). As a consequence, we extend
the main result of \cite{JR}, as follows. The next statement solves a problem raised
in the latter paper.

\begin{corollary} Let $G$ be a locally compact abelian group and let $L^p(G)$ denote the
corresponding $L^p$-spaces with respect to a Haar measure. Let $\nu$ be a probability
measure on $G$ and let
$T\colon L^1(G)+L^\infty(G)\to L^1(G)+L^\infty(G)$ be the associated convolution
operator,
$$
T(x)\,=\, \nu*x.
$$
Assume that there exists a constant $K>0$ such that
$\vert 1-\widehat{\nu}(s)\vert\leq K (1-\vert\widehat{\nu}(s)\vert)\,$ for any $s\in\widehat{G}$.
Then we have an estimate
$$
\bignorm{\bigl(T^n (x)\bigr)_{n\geq 0}}_{L^p(v^q)}\,\lesssim\,\norm{x}_p,\qquad x\in L^p(\Omega),
$$
for any $1<p<\infty$ and any $2<q<\infty$.
\end{corollary}

\begin{proof} Regard $T$ as an $L^2$-operator. By Fourier analysis, its spectrum is equal to the
essential range of $\widehat{\nu}$. It therefore follows from the assumption and Lemma \ref{3Normal1}
that the operator $T\colon L^2(G)\to L^2(G)$ is analytic. Hence $T$ satisfies the assumptions of
Corollary \ref{3Diff3}, which yields the result.
\end{proof}

\begin{remark}
For an operator $T\in B(L^2(\Omega))$,
let $W(T)=\{\langle T(x),x\rangle\, :\, \norm{x}_2=1\}$ denote the numerical range.
We recall that $W(T)$ is a compact convex set and that $\sigma(T)\subset W(T)$.
Assume that there exists $\gamma\in(0,\frac{\pi}{2})$ such that $W(T)\subset B_\gamma$
(which is a stronger condition than (\ref{3Stolz1}) or (\ref{3Stolz2})). Then according
to \cite{DD}, there exists a constant $C>0$ such that
$$
\norm{f(T)}\leq\,C\sup\bigl\{\vert f(z)\vert\, :\, z\in B_\gamma\bigr\}
$$
for any polynomial $f$. Arguing as in the proof of Lemma \ref{3Normal1}, we deduce
that $T$ is an analytic power bounded operator.

Consequently if $T\colon L^2(\Omega)\to L^2(\Omega)$ is contractively regular and
$W(T)\subset B_\gamma$ for some $\gamma\in(0,\frac{\pi}{2})$, then it satisfies
(\ref{3Tn}).
\end{remark}

\end{document}